\documentclass[12pt,letterpaper]{article}
\usepackage{amsmath,amssymb,amsthm}
\usepackage[margin=3cm]{geometry}
\usepackage{microtype}
\usepackage{url}
\usepackage{tikz}
\usepackage{enumitem}
\usepackage[normalem]{ulem} 

\tikzstyle{B}=[draw,circle, fill=black, minimum size=5pt,inner sep=0pt]
\tikzstyle{W}=[draw,circle, fill=white, minimum size=8pt,inner sep=0pt]
\tikzstyle{R}=[draw,circle, fill=red, minimum size=8pt,inner sep=0pt]
\tikzstyle{A}=[draw,circle, fill=blue, minimum size=8pt,inner sep=0pt]
\tikzstyle{G}=[draw,circle, fill=green, minimum size=8pt,inner sep=0pt]
\tikzstyle{Z}=[draw,circle, fill=black, minimum size=0pt,inner sep=0pt]

\newtheorem{thm}{Theorem}[section]

\newtheorem{lemma}[thm]{Lemma}

\theoremstyle{definition}
\newtheorem{prob}[thm]{Problem}

\theoremstyle{definition}
\newtheorem{question}[thm]{Question}

\theoremstyle{definition}

\theoremstyle{definition}

\theoremstyle{definition}

\theoremstyle{definition}

\theoremstyle{remark}
\newtheorem{remark}[thm]{Remark}

\usepackage[colorlinks=true, linkcolor=blue]{hyperref}

\begin{document}

\title{Independent set and matching permutations}

\author{Taylor Ball, David Galvin, Catherine Hyry and Kyle Weingartner\thanks{Department of Mathematics, University of Notre Dame, Notre Dame IN 46556; {\tt dgalvin1@nd.edu}. Galvin supported in part by the Simons foundation. Hyry and Weingartner supported in part by NSF grant DMS 1547292.}}

\date{\today}

\maketitle
 
\begin{abstract}
Let $G$ be a graph $G$ whose largest independent set has size $m$. A permutation $\pi$ of $\{1, \ldots, m\}$ is an {\em independent set permutation} of $G$ if 
$$
a_{\pi(1)}(G) \leq a_{\pi(2)}(G) \leq \cdots \leq a_{\pi(m)}(G)  
$$
where $a_k(G)$ is the number of independent sets of size $k$ in $G$. In 1987 Alavi, Malde, Schwenk and Erd\H{o}s proved that every permutation of $\{1, \ldots, m\}$ is an independent set permutation of some graph with $\alpha(G)=m$, i.e. with largest independent set having size $m$. They raised the question of determining, for each $m$, the smallest number $f(m)$ such that every permutation of $\{1, \ldots, m\}$ is an independent set permutation of some graph with $\alpha(G)=m$ and with at most $f(m)$ vertices, and they gave an upper bound on $f(m)$ of roughly $m^{2m}$. Here we settle the question,  determining $f(m)=m^m$, and make progress on a related question, that of determining the smallest order such that every permutation of $\{1, \ldots, m\}$ is the {\em unique} independent set permutation of some graph of at most that order. More generally we consider an extension of independent set permutations to weak orders, and extend Alavi et al.'s main result to show that every weak order on $\{1, \ldots, m\}$ can be realized by the independent set sequence of some graph with $\alpha(G)=m$ and with at most $m^{m+2}$ vertices.

Alavi et al.\! also considered {\em matching permutations}, defined analogously to independent set permutations. They observed that not every permutation of $\{1,\ldots,m\}$ is a matching permutation of some graph with largest matching having size $m$, putting an upper bound of $2^{m-1}$ on the number of matching permutations of $\{1,\ldots,m\}$. Confirming their speculation that this upper bound is not tight, we improve it to $O(2^m/\sqrt{m})$.

\end{abstract}

\medskip


\section{Introduction}

To a real sequence $a_1, a_2, \ldots, a_m$ we can associate a permutation $\pi$ of $[m]:=\{1, \ldots, m\}$, which gives information about the shape of the histogram of the sequence, via
\begin{equation} \label{perm-assoc}
a_{\pi(1)} \leq a_{\pi(2)} \leq \cdots \leq a_{\pi(m)}.
\end{equation}
If there are some repetitions among the $a_i$ then $\pi$ is not unique. For example, the sequence $(5,10,10,5,1)$ has associated with it each of the sequences $51423$, $54123$, $51432$ and $54132$. (Here and elsewhere we present permutations in one-line notation, so for example $51423$ represents the permutation $\pi$ with $\pi(1)=5$, $\pi(2)=1$, et cetera.) 

This association was introduced by Alavi, Malde, Schwenk and Erd\H{o}s in \cite{AMSE}, where they proposed using it to investigate sequences associated with graphs. For example, let $G$ be a (simple, finite) graph with $\alpha(G)=m$, that is, whose largest independent set (set of mutually non-adjacent vertices) has size $m$. The {\em independent set sequence} of $G$ is the sequence $(i_k(G))_{k=1}^m$ where $i_k(G)$ is the number of independent sets of size $k$ in $G$. Say that $\pi$ is an {\em independent set permutation} of $G$ if $\pi$ is one of the permutations that can be associated to the independent set sequence of $G$ via (\ref{perm-assoc}). (We do not consider $i_0(G)$, as it equals $1$ for every $G$.) 

The main theorem of \cite{AMSE} is that all $m!$ permutations of $[m]$ are independent set permutations.
\begin{thm} \label{thm-AMSE-main} \cite{AMSE}
Given $m \geq 1$ and a permutation $\pi$ of $[m]$, there is a graph $G$ with $\alpha(G)=m$ and with
\begin{equation} \label{AMSE-main-inq}
i_{\pi(1)}(G) < i_{\pi(2)}(G) < \cdots < i_{\pi(m)}(G).
\end{equation}
\end{thm}    
In the language of \cite{AMSE} the independent set sequence of a graph is {\em unconstrained} --- it can exhibit arbitrary patterns of rises and falls.

For a permutation $\pi$ denote by $g(\pi)$ the minimum order (number of vertices) over all graphs $G$ for which $\pi$ is an independent set permutation of $G$, and for each $m$ denote by $f(m)$ the maximum, over all permutations $\pi$ of $[m]$, of $g(\pi)$. Alavi et al. showed that $f(m)$ is at most roughly $m^{2m+1}$ (they did not calculate their upper bound explicitly). They speculated that $f(m)\geq m^m$, and proposed the question of determining $f(m)$.
\begin{prob} \label{AMSE-ind-Q}
\cite[Problem 1]{AMSE} Determine the smallest order large enough to realize every permutation of order $m$ as the sorted indices of the vertex independent set sequence of some graph.
\end{prob}
Our first result settles this question exactly.
\begin{thm} \label{thm-ind-set-main}
(Part 1, $f(m)\leq m^m$) For each $m \geq 1$ there is a graph $G_m$ on $m^m$ vertices with $\alpha(G)=m$ and with
\begin{equation} \label{all-equal}
i_1(G_m)=i_2(G_m)=\cdots= i_m(G_m)=m^m.    
\end{equation}
(Part 2, $f(m)\geq m^m$) On the other hand, if $\alpha(G)=m$ and $i_m(G) < m^m$ then $i_m(G) < i_{m-1}(G)$. 
\end{thm}
Note that Part 1 of Theorem \ref{thm-ind-set-main} immediately implies that $f(m)\geq m^m$, since for every permutation $\pi$ of $[m]$, $\pi$ is an independent set permutation of $G_m$. To see that Part 2 implies $f(m)\geq m^m$, consider any permutation of the form 
$$
\cdots(m-1)\cdots m \cdots 1 \cdots.
$$
Since $m$ appears later in the permutation than $m-1$, for this to be an independent set permutation of some graph $G$ requires $i_m(G) \geq i_{m-1}(G)$, and so, by (the contrapositive of) Part 2, $i_m(G) \geq m^m$. But then since $1$ appears later in the permutation than $m$, this further requires $i_1(G)\geq m^m$, so $G$ must have at least $m^m$ vertices.

Our proof that $f(m)\geq m^m$ follows almost immediately from a result of Fisher and Ryan \cite{FisherRyan} on the monotonicity of a sequence related to the independent set sequence. Our construction of $G_m$, to establish $f(m)\leq m^m$, follows the same general scheme introduced in \cite{AMSE}. There, it is shown how to construct a graph $G$ with $\alpha(G)=m$, with $i_k(G)$ being a sum. The first term of the sum is $\pi^{-1}(k)T$ (for some arbitrary constant $T$), and for $T$ sufficiently large the sum of the remaining terms can be bounded above by $T$. This puts $i_k(G)$ in the interval $[\pi^{-1}(k)T, (\pi^{-1}(k)+1)T)$, and so $\pi$ is a (actually, the unique) independent set permutation of $G$. (We describe this construction in more detail in Section \ref{sec-isperm}). We obtain $f(m)\leq m^m$ by carefully carrying out the construction in a way that allows perfect control over the lower order terms in the sum.

It is worth noting here a difference between (\ref{perm-assoc}) (which allows different terms of the sequence to have the same value) and (\ref{AMSE-main-inq}) (which does not). It is quite natural to ask what happens in Problem \ref{AMSE-ind-Q} when we require that the permutations associated with independent set sequences be unique. 
\begin{prob} \label{prob-AMSE-e}
Determine, for each $m \geq 1$, the smallest $M$ such that for every permutation $\pi$ of $[m]$ there is a graph $G$ of order at most $M$ with $\alpha(G)=m$ and with
$$
i_{\pi(1)} < i_{\pi(2)} < \cdots < i_{\pi(m)}.
$$ 
\end{prob}
In \cite{AMSE} the comment is made that Problem \ref{AMSE-ind-Q} ``is likely to remain exceeding difficult''. Given the surrounding discussion in \cite{AMSE}, it seems likely that the authors were implicitly thinking about Problem \ref{prob-AMSE-e} when they made this comment. While we do not have an exact answer to Problem \ref{prob-AMSE-e}, we are able to extend the approach used in Theorem \ref{thm-ind-set-main} to obtain bounds for $M$ in Problem \ref{prob-AMSE-e} that are significantly better than those implicit in \cite{AMSE} (see Theorem \ref{thm-ind-set-weak} below).

To a real sequence $a_1, a_2, \ldots, a_m$ we can associate a unique weak order (an ordered partition $(B_1,\ldots, B_\ell)$ of $[m]$ into non-empty blocks)
via $B_i=\{b_{i1}, b_{i2}, \ldots\}$, where
$$
a_{b_{11}} = a_{b_{12}} = \cdots <  a_{b_{21}} = a_{b_{22}} = \cdots < \cdots < a_{b_{\ell1}} = a_{b_{\ell2}} = \cdots.
$$
For example the sequence $(4,6,4,1)$ (the independent set sequence of the edgeless graph on four vertices) induces the weak order $B_1=\{4\}$, $B_2=\{1,3\}$, $B_3=\{2\}$. Theorem \ref{thm-AMSE-main} says that every weak order in which all blocks are singletons is the weak order induced by some graph, while Part 1 of Theorem \ref{thm-ind-set-main} says the same for the weak order with a single block.
\begin{thm} \label{thm-ind-set-weak}
For $m \geq 1$, for every weak order $w$ on $[m]$ there is a graph $G$ with $\alpha(G)=m$, and with fewer than $m^{m+2}$ vertices, which induces $w$.
\end{thm}
So although there are many more weak orders on $[m]$ than there are permutations --- $(1/2)m!(\log_2e)^{m+1}$ (see e.g. \cite{Barthelemy}) as opposed to $m!$ --- it does not take too many more vertices to induce them all. Note also that by Theorem \ref{thm-ind-set-main}, any weak order on $[m]$ that has $m-1$ and $m$ in the same block, and $1$ in a block with a higher index, cannot be induced by a graph with $m^m$ or fewer vertices. The analog of Problem \ref{AMSE-ind-Q} for weak orders --- where in the range $(m^m, m^{m+2})$ is the smallest order sufficient to realize every weak order on $[m]$? --- remains open.

\medskip

Alavi et al. also considered the {\em edge independent set sequence} or {\em matching sequence} of a graph. Let ${\mathcal M}_n$ denote the set of graphs with $\nu(G)=n$, that is, whose largest matching (set of edges no two sharing a vertex) has $n$ edges.
The matching sequence of $G \in {\mathcal M}_n$ is $(m_k(G))_{k=1}^n$ where $m_k(G)$ is the number of matchings in $G$ with $k$ edges. Say that $\pi$ is a {\em matching permutation} of $G$ if $\pi$ is one of the permutations that can be associated to the matching sequence of $G$ via (\ref{perm-assoc}). (Note that throughout our discussion of matchings, we will only consider simple graphs.)

In contrast to independent set permutations, there are permutations that are not the matching permutation of any graph. Indeed, Schwenk \cite{Schwenk} showed that the matching sequence of any graph $G \in {\mathcal M}_n$ is unimodal in the strong sense that for some $k$,
$$
m_1(G) < m_2(G) < \cdots < m_k(G) \geq m_{k+1}(G) > m_{k+2}(G) > \cdots > m_n(G).
$$
It follows that the permutations of $[n]$ that can be the matching permutations of a graph in ${\mathcal M}_n$ must have
\begin{equation}\label{unimodal}
\begin{array}{c}     
     \pi^{-1}(1)<\pi^{-1}(2)<\cdots < \pi^{-1}(k-1)\\~\mbox{and}\\
     \pi^{-1}(n)< \pi^{-1}(n-1)< \cdots < \pi^{-1}(k+1),
\end{array}
\end{equation}  
where $k=\pi(n)$.
(This restriction on $\pi$ can also be deduced from the  real-rootedness of the matching polynomial, first established by Heilmann and Lieb \cite{HL}.) Following Alavi et al., we refer to permutations satisfying (\ref{unimodal}) as {\em unimodal} permutations.

There are $\sum_{k=0}^{n-1} \binom{n-1}{k} = 2^{n-1}$
unimodal permutations of $[n]$. To see this, note that to construct a unimodal permutation we first select $k=\pi(n)$, which must appear as the last entry of the permutation in one-line notation, and then select the $k-1$ locations (from among the first $n-1$) where $1, \ldots, k-1$ appear; this completely determines the permutation since, as observed in (\ref{unimodal}) above, the entries $1$ through $k$ must appear in $\pi$ in ascending order, while the entries $k+1$ through $n$ must appear in descending order. So, writing $M_n$ for the set of permutations $\pi$ that are the matching permutations of some graph in ${\mathcal M}_n$, we have
$M_n \leq 2^{n-1}$.
This bound was observed in \cite{AMSE}, where the following problem was posed.
\begin{prob} \label{AMSE-matching-q}
\cite[Problem 2]{AMSE} 
Characterize the permutations realized by the edge independence sequence. In particular, can all $2^{n-1}$ unimodal permutations of $[n]$ be realized?   
\end{prob}
We do not address the characterization problem, but our next result answers the particular question: a vanishing proportion of unimodal permutations are the matching permutations of some graph.
\begin{thm} \label{thm-ub-on-matching}
We have $M_n=o(2^n)$. More precisely, there is a constant $c$ such that for $n \geq 1$
\begin{equation} \label{eq-up-on-count}
M_n \leq \frac{c2^n}{\sqrt{n}}.
\end{equation}
\end{thm}
In the other direction, the perfect matching with $n$ edges gives a lower bound on $M_n$ of $2^{\lfloor (n-1)/2 \rfloor}$. Indeed, the matching sequence of the perfect matching with $n$ edges is $\left(\binom{n}{k}\right)_{k=1}^n$, which has $\lfloor (n-1)/2 \rfloor$ pairs of equal terms ($\binom{n}{1}=\binom{n}{n-1}$, $\binom{n}{2}=\binom{n}{n-2}$, et cetera), leading to $2^{\lfloor (n-1)/2 \rfloor}$ associated permutations of $[n]$. We can improve this by an additive term of $\Omega(n)$, but we do not give the details here.

\medskip

We give the proofs of our results concerning independent set permutations and weak orders in Section \ref{sec-isperm}, and address matching permutations in Section \ref{sec-matperm}. We end with some questions and comments in Section \ref{sec-questions}.

\section{Independent set permutations} \label{sec-isperm}

We begin with the proof of Part 2 of Theorem \ref{thm-ind-set-main}, $f(m)\geq m^m$. This turns out to come almost immediately from a theorem of Fisher and Ryan \cite{FisherRyan}, a result which they remark ``brings order into [the] chaos'' of the independent set sequence observed by Alavi et al..
\begin{thm} \label{thm-FisherRyan}
For any graph $G$ with $\alpha(G)=m$, we have
$$
\left(\frac{i_1(G)}{\binom{m}{1}}\right)^\frac{1}{1} \geq \left(\frac{i_2(G)}{\binom{m}{2}}\right)^\frac{1}{2} \geq  \cdots \geq \left(\frac{i_{m-1}(G)}{\binom{m}{m-1}}\right)^\frac{1}{m-1} \geq \left(\frac{i_m(G)}{\binom{m}{m}}\right)^\frac{1}{m}.
$$
\end{thm}
The last inequality above (which is all we need) says that $m^m i_m(G)^{m-1} \leq i_{m-1}(G)^m$. If also $i_m(G) < m^m$ then this implies that $i_m(G)^m < i_{m-1}(G)^m$, or $i_m(G)<i_{m-1}(G)$, as claimed.

\begin{remark}
In an earlier version of this paper \cite{BHGW-arXiv} we obtained Part 2 of Theorem \ref{thm-ind-set-main} by combining results of Frankl, F\"uredi and Kalai \cite{FFK} and Frohmader \cite{Frohmader} on Kruskal-Katona type theorems for colored (or balanced) flag complexes. Invoking Theorem \ref{thm-FisherRyan} (whose short proof does not require consideration of flag complexes) leads to a considerably more direct proof.   
\end{remark}

\medskip

We now move on to the proof of Part 1 of Theorem \ref{thm-ind-set-main}, $f(m)\leq m^m$. We begin with an outline of the construction, which is very similar to one described in \cite{AMSE}. Recall that our goal is to construct a graph $G_m$ with $\alpha(G)=m$ that has $m^m$ independent sets of size $k$ for each $k \in [m]$. A key idea that we use throughout is the effect of the join operation on independent set sequences. For a collection $\{G_j : j \in J\}$ of graphs, denote by $\oplus_{j \in J} G_j$ the graph consisting of a union of disjoint copies of the $G_j$, with every vertex in each $G_j$ adjacent to every vertex in $G_{j'}$ for each $j' \neq j$ --- the {\em mutual join} of the $G_j$. The effect of $\oplus$ on independent set sequences is additive: if $G=\oplus_{j \in J} G_j$ then for $k \geq 1$,
\begin{equation} \label{join}
i_k(G) = \sum_{j \in J} i_k(G_j),
\end{equation}
because no independent set in $G$ can have vertices in two different $G_j$'s. We will use (\ref{join}) repeatedly in the sequel, usually without comment.

Given a permutation $\pi$ of $[m]$, to construct a graph $G$ satisfying (\ref{AMSE-main-inq}) (i.e., $i_{\pi(1)}(G) < \cdots < i_{\pi(m)}(G)$)  Alavi et al. \cite{AMSE} consider a graph of the form 
$$
G_{\pi}:=\oplus_{k=1}^m kK_{n_k},
$$
where $n_k = (\pi^{-1}(k)T)^{1/k}$ for some large integer $T$, and where $kK_{n_k}$ denotes $k$ vertex disjoint copies of the complete graph $K_{n_k}$ on $n_k$ vertices. By (\ref{join}) we have
\begin{equation} \label{AMSE-sum}
i_k(G_{\pi}) = \pi^{-1}(k)T + \sum_{j=k+1}^m \binom{j}{k}(\pi^{-1}(j)T)^\frac{k}{j}. 
\end{equation}
Here the term $\pi^{-1}(k)T$ is the count of independent sets of size $k$ in $kK_{n_k}$, and for $j>k$ the summand $\binom{j}{k}(\pi^{-1}(j)T)^\frac{k}{j}$ counts independent sets of size $k$ in $jK_{n_j}$; there are no independent sets of size $k$ in any $jK_{n_j}$ for $j < k$. For $k < m$ we have
$$
\sum_{j=k+1}^m \binom{j}{k}(\pi^{-1}(j)T)^\frac{k}{j} \leq T^\frac{k}{k+1} \sum_{j=k+1}^m \binom{j}{k}\pi^{-1}(j)^\frac{k}{j} \leq T^\frac{m}{m+1} \left(m2^mm^\frac{m}{m+1}\right).
$$
For large enough $T=T(m)$ the last expression above is strictly smaller than $T$, so that from (\ref{AMSE-sum}) we get $\pi^{-1}(k)T \leq i_k(G_{\pi}) < (\pi^{-1}(k)+1)T$. This inequality also holds when $k=m$ (in which case the summation in (\ref{AMSE-sum}) is empty). From all this (\ref{AMSE-main-inq}) follows.

To more carefully control the sum in (\ref{AMSE-sum}), and allow us to construct a graph $G_m$ with $m^m$ independent sets of all sizes from $1$ to $m$, we modify this construction. Before doing so, we give some intuition.
 
The graph $G_0:=mK_m$ has $\alpha(G_0)=m$, $i_m(G_0)=i_{m-1}(G_0)=m^m$, and $i_k(G_0) = \binom{m}{k}m^k < m^m$ for $k < m-1$. We need to increase the count of independent sets of size $m-2$ by
$$
m^m-\binom{m}{2}m^{m-2}=m^{m-2}\left(m^2-\binom{m}{2}\right):=a^{(m)}_2m^{m-2},
$$
without changing the number of independent sets of sizes $m$ or $m-1$. By (\ref{join}), the graph $G_2:=\oplus_{i=1}^{a^{(m)}_2} (m-2)K_m$ (the mutual join of $a^{(m)}_2$ copies of $(m-2)K_m$) has $i_{m-2}(G_2)=a^{(m)}_2m^{m-2}$, and also has $i_m(G_2)=i_{m-1}(G_2)=0$. Hence, again by (\ref{join}), $\alpha(G_0 \oplus G_2)=m$, $i_m(G_0 \oplus G_2)=i_{m-1}(G_0 \oplus G_2)=i_{m-2}(G_0 \oplus G_2)=m^m$, and $i_{m-3}(G_0 \oplus G_2)=\binom{m}{3}m^{m-3} + a^{(m)}_2(m-2)m^{m-3}$. We need to add
$$
m^{m-3}\left(m^3 - \binom{m}{3} - a^{(m)}_2(m-2)\right) := a^{(m)}_3m^{m-3}
$$
independent sets of size $m-3$ (without adding any independent sets of sizes $m, m-1$ or $m-2$). We achieve this by setting
$$
G_3:= \oplus_{i=1}^{a^{(m)}_3} (m-3)K_m
$$
and considering $G_0\oplus G_2 \oplus G_3$. (Note that $a^{(m)}_3 \geq 0$, being a cubic in $m$ with non-negative coefficients.)

We continue in this manner until we reach a graph which satisfies (\ref{all-equal}), which we declare to be $G_m$. We have to check that at no point, while fixing the number of independent sets of size $k$ to be $m^m$, do we cause the number of independent sets of size $j$ to be greater than $m^m$, for some $1 \leq j < k$. This check is the main point of the formal proof of Theorem \ref{thm-ind-set-main}, Part 1.

\medskip

\begin{proof} (Theorem \ref{thm-ind-set-main}, Part 1)
For $m \geq 1$, define a sequence $(a^{(m)}_0, a^{(m)}_1, \ldots, a^{(m)}_{m-1})$ via
\begin{equation} \label{recurrence}
m^k = a^{(m)}_0\binom{m}{k} + a^{(m)}_1\binom{m-1}{k-1} + \cdots + a^{(m)}_{k-1}\binom{m-(k-1)}{1} + a^{(m)}_k\binom{m-k}{0}
\end{equation}
for $k=0, \ldots, m-1$. Note that the $m$ relations in (\ref{recurrence}) do indeed uniquely determine the $a^{(m)}_k$: first taking $k=0$ forces $a^{(m)}_0=1$; then taking $k=1$ further forces $a^{(m)}_1=0$; then taking $k=2$ forces $a^{(m)}_2=m^2-\binom{m}{2}$, and so on. The motivation behind this definition as follows: we will go through an iterative procedure (the one described above) to set the number of independent sets of each size to be $m^m$, starting with independent sets of size $m$, and working down. When we come to fix the number of independent sets of size $m-k$ to be $m^m$, it will turn out that we need to add $a^{(m)}_k m^{m-k}$ such, which we will achieve by successively joining $a^{(m)}_k$ copies of $(m-k)K_m$ to what has thus far been constructed.   Evidently each $a^{(m)}_i$ is an integer; but in fact $a^{(m)}_i \geq 0$, as we now show.

For $m=1$ the sequence consists of the single term $a^1_0=1$, and for $m=2$ the sequence is $(1,0)$. So consider $m \geq 3$. We will show, for each such $m$, that $a^{(m)}_k \geq 0$ for all $0 \leq k \leq m-1$. We evidently have $a^{(m)}_0=1$. Now consider a $k$ with $1 \leq k \leq m-1$. Starting by multiplying both sides of the $k-1$ instance of (\ref{recurrence}) by $m$, and with the rest of the steps justified below, we have
\begin{eqnarray*}
m^k & = & a^{(m)}_0 m\binom{m}{k-1} + a^{(m)}_1 m\binom{m-1}{k-2} + \cdots + a^{(m)}_{k-1} m\binom{m-(k-1)}{0} \\
& \geq & a^{(m)}_0\binom{m}{k} + a^{(m)}_1\binom{m-1}{k-1} + \cdots + a^{(m)}_{k-1}\binom{m-(k-1)}{1} \\
& = & m^k - a^{(m)}_k,
\end{eqnarray*}
so $a^{(m)}_k\geq 0$.
The first inequality uses
$$
m\binom{m-j}{k-1-j} \geq \binom{m-j}{k-j},
$$
valid for $m \geq 3$, $k \in \{1, \ldots, m-1\}$ and $j \in \{0, \ldots, k-1\}$, and the second equality uses  (\ref{recurrence}).

Now consider the graph
$G_m = \oplus_{k=0}^{m-1} G_k$ where $G_k = \oplus_{j=1}^{a^{(m)}_k} (m-k)K_m$. 
We have $\alpha(G_m)=m$ and, for each $k \in \{0 \ldots, m-1\}$
\begin{eqnarray*}
i_{m-k}(G_m) & = & a^{(m)}_0\binom{m}{k}m^{m-k} + a^{(m)}_1\binom{m-1}{k-1}m^{m-k} + \cdots  + a^{(m)}_k\binom{m-k}{0}m^{m-k} \\
& = & m^{m-k}\left(a^{(m)}_0\binom{m}{k} + a^{(m)}_1\binom{m-1}{k-1} + \cdots +  a^{(m)}_k\binom{m-k}{0}\right) \\
& = & m^m,
\end{eqnarray*}
the last equality by (\ref{recurrence}). The main points of the calculation above are that the only parts of $G_m$ that contribute to $i_{m-k}(G_m)$ are those of the form $aK_m$ for $a\geq m-k$, and that
$$
i_{m-k}(aK_m) = \binom{a}{m-k}m^{m-k} = \binom{m-(m-a)}{k-(m-a)}m^{m-k}.
$$
\end{proof}

\medskip

We now turn to the proof of Theorem \ref{thm-ind-set-weak}, concerning weak orders. The case $m=1$ is trivial, and $m=2$ is easy: the three weak orders on $[2]$ are achieved by $2K_1$, $2K_2$ and $K_3 \cup K_2$. So from here on we assume $m \geq 3$.

We will construct
\begin{itemize}
\item a graph $H_1$ with $m^m+m^{m-1}$ vertices, with $m^m$ independent sets of each size in $\{2, \ldots, m\}$, $m^m+m^{m-1}$ independent sets of size $1$, and with $\alpha(H_1)=m$;
\item a graph $H_m$ with $2m^m-m^{m-1}$ vertices, with $2m^m-m^{m-1}$ independent sets of each size in $\{1, \ldots, m-1\}$, $2m^m$ independent sets of size $m$, and with $\alpha(H_m)=m$;
\item and for each $k \in \{2, \ldots, m-1\}$, a graph $H_k$ with $m^m$ vertices, with $m^m$ independent sets of each size in $\{1, \ldots, m\}\setminus\{k\}$, with $m^m+m^{m-1}$ independent sets of size $k$, and with $\alpha(H_m)=m$.
\end{itemize}
The main point here is that for each $k$ there is a value $s(k)$ such that $H_k$ has $s(k)$ independent sets of all sizes except $k$, and has $s(k)+m^{m-1}$ independent sets of size $k$ (specifically $s(k)=m^m$ for $k \neq m$ and $s(m)=2m^m-m^{m-1}$). 

Let $w=(B_1, \ldots, B_\ell)$ be a weak order on $[m]$. Construct a graph $H(w)$ as follows: $H(w)$ is the mutual join of 
\begin{itemize}
\item one copy of $G_m$ for each $k \in B_1$ (here and later, $G_m$ is the graph from Theorem \ref{thm-ind-set-main}, Part 1; recall that it is a graph on $m^m$ vertices, with largest independent set having size $m$, and with $m^m$ independent sets of size $k$ for each $k=1, \ldots, m$); 
\item one copy of $H_k$ for each $k \in B_2$; 
\item and in general $j-1$ copies of $H_k$ for each $k \in B_j$. 
\end{itemize}
For $t \in B_j$, for any $1 \leq j \leq \ell$, we have
\begin{equation} \label{eq-blip}
i_t(H(w)) = \left(m^m|B_1| + \sum_{k \in B_2} s(k) + 2\sum_{k \in B_3} s(k) + \cdots + \sum_{k \in B_\ell} (\ell-1)s(k)\right) + (j-1)m^{m-1}. 
\end{equation}
Indeed, $H(w)$ has $m^m$ independent sets of size $t$, coming from each of the $|B_1|$ copies of $G_m$ in the construction; for each $k \in B_2$ it has a further $s(k)+{\bf 1}_{\{t=k\}}m^{m-1}$ independent sets of size $t$, coming from the $H_k$; and in general, for each $k \in B_j$ ($1 \leq j \leq \ell$) it has a further $(j-1)\left(s(k)+{\bf 1}_{\{t=k\}}m^{m-1}\right)$ independent sets of size $t$, coming from the $j-1$ copies of $H_k$. Summing all these up, and noting that ${\bf 1}_{\{t=k\}}$ will take the value $1$ at most once (for that $j$ for which $t \in B_j$, if $j >1$), we obtain (\ref{eq-blip}). 

Note that the term in parentheses in (\ref{eq-blip}) depends only on the weak order $w$, and in particular is independent of $t$; let this term be denoted by $c(w)$. We have that $H(w)$ has
\begin{itemize}
\item $c(w)$ independent sets of size $t$ for each $t \in B_1$;
\item $c(w)+ m^{m-1}$ independent sets of size $t$ for each $t \in B_2$;
\item and in general, $c(w)+ (j-1)m^{m-1}$ independent sets of size $t$ for each $t \in B_j$, for $1 \leq j \leq \ell$,
\end{itemize} 
and so the weak order induced by $H(w)$ is indeed $w$.

Among the $H_k$ none has more than $2m^m-m^{m-1}$ vertices, so the order of $H(w)$ is at most
\begin{equation} \label{verts-in-Hw}
m^m|B_1| + (|B_2|+2|B_3|+\cdots + (\ell-1)|B_\ell|)(2m^m-m^{m-1}).
\end{equation} 
If any of the $B_i$'s has size at least $2$, then the quantity in (\ref{verts-in-Hw}) can be increased by replacing $|B_i|$ with $|B_i|-1$ and $|B_{i+1}|$ with $|B_{i+1}|+1$ (creating a new, $(\ell+1)$st, block if $i=\ell$). It follows that subject to the constraints $\sum_i |B_i| =m$ and $|B_i| \geq 1$, the quantity in (\ref{verts-in-Hw}) is maximized by
$$
m^m + (1+2+\ldots + (m-1))(2m^m+m^{m-1}) < m^{m+2}.
$$
This gives Theorem \ref{thm-ind-set-weak}; so our goal (which occupies the rest of the section) is to construct $H_k$, for $k \in\{1, \ldots,m\}$.

\medskip

In the proof of Theorem \ref{thm-ind-set-main}, we required $a^{(m)}_k \geq 0$.  To construct $H_k$, we need a better bound.
\begin{lemma} \label{lem-a_k-large}
For $k \geq 2$ (and $m \geq 3$), $a^{(m)}_k \geq m^{k-1}$.
\end{lemma}

\begin{proof} 
We will use an explicit expression for the $a^{(m)}_k$. It will be convenient in what follows to extend the sequence $(a^{(m)}_0, \ldots, a^{(m)}_{m-1})$ to $(a^{(m)}_0, \ldots, a^{(m)}_m)$, by using (\ref{recurrence}) to also define $a^{(m)}_m$.

Let $\vec{a}^{(m)}$ be the column vector with $a^{(m)}_j$ in the $j$th position (with the positions indexed from $0$ to $m$), and $\vec{m}$ the column vector with $m^j$ in the $j$th position; so 
$$
\vec{a}^{(m)} = \left[a^{(m)}_0~a^{(m)}_1~\cdots~a^{(m)}_m\right]^{\tt T}~~~\mbox{and}~~~\vec{m} = \left[1~m~\cdots~m^m\right]^{\tt T}.
$$
From (\ref{recurrence}) we have $M\vec{a}^{(m)}=\vec{m}$ where $M$ is the $(m+1)$ by $(m+1)$ matrix with $\binom{m-j}{i-j}$ in the $(i,j)$ position (rows and columns indexed from $0$). Here we understand $\binom{n}{c}$ to be $0$ for negative $c$. Since $M$ is lower triangular with $1$'s down the diagonal it is invertible, and it is well known that $M^{-1}$ is the matrix with $(-1)^{i-j}\binom{m-j}{i-j}$ in the $(i,j)$ position (see for example \cite{CallVelleman}). To illustrate this fact, and the structure of $M$ and $M^{-1}$, consider $M^{-1}$ in the case $m=4$:
$$
\left[\begin{array}{ccccc}
\binom{4}{0} & 0 & 0 & 0 & 0 \\[3pt]
\binom{4}{1} & \binom{3}{0} & 0 & 0 & 0 \\[3pt]
\binom{4}{2} & \binom{3}{1} & \binom{2}{0} & 0 & 0 \\[3pt]
\binom{4}{3} & \binom{3}{2} & \binom{2}{1} & \binom{1}{0} & 0 \\[3pt]
\binom{4}{4} & \binom{3}{3} & \binom{2}{2} & \binom{1}{1} & \binom{0}{0}
\end{array}\right]^{-1} = \left[\begin{array}{ccccc}
1 & 0 & 0 & 0 & 0 \\
4 & 1 & 0 & 0 & 0 \\
6 & 3 & 1 & 0 & 0 \\
4 & 3 & 2 & 1 & 0 \\
1 & 1 & 1 & 1 & 1
\end{array}\right]^{-1} = \left[\begin{array}{ccccc}
1 & 0 & 0 & 0 & 0 \\
-4 & 1 & 0 & 0 & 0 \\
6 & -3 & 1 & 0 & 0 \\
-4 & 3 & -2 & 1 & 0 \\
1 & -1 & 1 & -1 & 1
\end{array}\right].
$$
For completeness, we provide a proof that $M^{-1}$ is as claimed. Consider the matrix $M\overline{M}$, where $\overline{M}$ has $(-1)^{i-j}\binom{m-j}{i-j}$ in the $(i,j)$ position. The $(k,\ell)$ entry of $M\overline{M}$ is clearly $0$ for $k<\ell$, and $1$ for $k=\ell$. For $\ell<k$ the $(k,\ell)$ entry is
\begin{eqnarray*}
\sum_{t=\ell}^k (-1)^{t-\ell} \binom{m-t}{k-t}\binom{m-\ell}{t-\ell} & = & (-1)^{\ell-k} \sum_{t=\ell}^k (-1)^{k-t} \frac{(m-t)!}{(k-t)!(m-k)!}\frac{(m-\ell)!}{(t-\ell)!(m-t)!} \\
& = & (-1)^{\ell-k}\binom{m-\ell}{m-k} \sum_{t=\ell}^k (-1)^{k-t} \frac{(k-\ell)!}{(k-t)!(t-\ell)!} \\
& = & (-1)^{\ell-k} \binom{m-\ell}{m-k} \sum_{t=\ell}^k (-1)^{k-t}\binom{k-\ell}{k-t} \\
& = & 0,
\end{eqnarray*}
the last equality following from the standard fact that the alternating sum of binomial coefficients is $0$. This shows that $M\overline{M}$ is the identity, and so the inverse of $M$ is as claimed.

Since $\vec{a}^{(m)}=M^{-1}\vec{m}$ we have
\begin{equation} \label{eq-a_k-exp}
a^{(m)}_k = m^k -m^{k-1}\binom{m-(k-1)}{1}+m^{k-2}\binom{m-(k-2)}{2} - \cdots + (-1)^k\binom{m}{k}.
\end{equation}
For $m \geq 3$ and $k \geq 2$, it is easily checked that the sequence
$$
m^k, ~m^{k-1}\binom{m-(k-1)}{1}, ~ m^{k-2}\binom{m-(k-2)}{2}, \ldots,~\binom{m}{k}
$$
is strictly decreasing. 
Lower bounding $a^{(m)}_k$ by the sum of the first two terms of the decreasing alternating sum on the right-hand side of (\ref{eq-a_k-exp}) we get 
$$
a^{(m)}_k > m^k - m^{k-1}\binom{m-(k-1)}{1} = (k-1)m^{k-1} \geq m^{k-1},
$$
as claimed.
\end{proof}

\medskip

Another tool we will need in the construction of the $H_k$ is the following easy observation.
\begin{lemma} \label{lem-obsv-dec}
If $\ell \leq n$ (with $\ell, n$ natural numbers), then the sequence
$$
n^\ell,~ \binom{\ell}{1}n^{\ell-1},~\binom{\ell}{2}n^{\ell-2}, \ldots, ~\binom{\ell}{\ell-1}n,~1
$$
is non-increasing. In fact it is strictly decreasing, except that when $\ell=n$ the first two terms are equal.
\end{lemma}

\medskip

Lemma \ref{lem-obsv-dec} gives an alternate justification that the procedure described in the proof of Theorem \ref{thm-ind-set-main} (the construction of $G_m$) is valid, which we now briefly describe, as it is relevant to the construction of the $H_k$. Recall that $G_m = \oplus_{k=0}^{m-1} G_k$ where $G_k = \oplus_{j=1}^{a^{(m)}_k} (m-k)K_m$ (the mutual join of $a^{(m)}_k$ copies of $(m-k)K_m$), where $a^{(m)}_k$ is as given by (\ref{recurrence}). The sequence $(i_m(G_0), i_{m-1}(G_0),\ldots,i_1(G))$ (which we will denote compactly by $(i_k(G_0))_{k=m}^1$) is $(\binom{m}{m-k}m^k)_{k=m}^1$ (recall $a^{(m)}_0=1$). This starts $(m^m, \ldots)$, is decreasing (by Lemma \ref{lem-obsv-dec}, with $(n,\ell)=(m,m)$), and its successive terms are integer multiples of $m^m, m^{m-1}, m^{m-2}, \ldots$. 

Now consider the sequence $(m^m-i_k(G_0))_{k=m-1}^1$, which represents the shortfall of the sequence $(i_k(G_0))_{k=m}^1$ from the goal sequence $(m^m)_{k=m}^1$ (in the shortfall, we have omitted the leading $0$, corresponding to $k=m$). This sequence is increasing, and its successive terms are integer multiples of $m^{m-1}, m^{m-2}, \ldots$. Its first term is $m^m-\binom{m}{1}m^{m-1}$, which by Lemma \ref{lem-obsv-dec} is a non-negative multiple of $m^{m-1}$ (and in fact by (\ref{recurrence}) is $a^{(m)}_1 m^{m-1}$). So, to $G_0$ we join the graph $G_1$, the mutual join of $a^{(m)}_1$ copies of $(m-1)K_m$. (It happens that $a^{(m)}_1=0$, but for the purposes of this discussion, all that matters is that it is non-negative). 

The sequence $(i_k(G_1))_{k=m-1}^1$ is $(a^{(m)}_1\binom{m-1}{(m-1)-k}m^k)_{k=m-1}^1$. By Lemma \ref{lem-obsv-dec}, with $(n,\ell)=(m,m-1)$, this is decreasing, and its successive terms are integer multiples of $m^{m-1}, m^{m-2}, \ldots$. It follows that the sequence $(m^m-i_k(G_0\oplus G_1))_{k=m-2}^1$ ---  representing the shortfall of the sequence $(i_k(G_0\oplus G_1))_{k=m}^1$ from the goal sequence $(m^m)_{k=m}^1$ (in the shortfall, we have now omitted the two leading $0$'s, corresponding to $k=m$ and $m-1$) --- is increasing, and its successive terms are integer multiples of $m^{m-2}, m^{m-3}, \ldots$. Its first term is $m^m-\binom{m}{2}m^{m-2}-a^{(m)}_1\binom{m-1}{1}m^{m-2}$, which by Lemma \ref{lem-obsv-dec} is a non-negative multiple of $m^{m-2}$ (and in fact by (\ref{recurrence}) is $a^{(m)}_2 m^{m-2}$). 

So, to $G_0\oplus G_1$ we join the graph $G_2$, the mutual join of $a^{(m)}_2$ copies of $(m-2)K_m$, which brings the number of independent sets of size $m-2$ up to the desired $m^m$, and leaves a shortfall sequence that is non-negative and (by an appropriate application of Lemma \ref{lem-obsv-dec}) increasing, with terms that are successively integer multiples of $m^{m-3}, m^{m-4}, \ldots$. This construction can be iteratively continued until $G_m$ is reached. 

We modify this process slightly to obtain $H_k$.

\medskip

\noindent {\bf Case 1}, $k=1$: Set $H_1=G_m \oplus K_{m^{m-1}}$. Note that this requires neither Lemma \ref{lem-a_k-large} nor Lemma \ref{lem-obsv-dec}.

\medskip

\noindent {\bf Case 2}, $k \neq m, 1$: 
At the moment when the number of independent sets of size $k$ has reached $m^m$, there are $m^m$ independent sets of all sizes at least $k$, while the sequence $(i_{k-1}(G), \ldots, i_1(G))$ (where $G$ is the graph constructed so far) is strictly decreasing, with $i_{k-1}(G)=m^m - a_{m-(k-1)}m^{k-1} \leq m^m-m^{m-1}$ (the equality coming from the proof of Theorem \ref{thm-ind-set-main}, Part 1, and the inequality using Lemma \ref{lem-a_k-large}), and with $i_j(G)$ a multiple of $m^j$.

Successively join $m^{m-k-1}$ copies of $kK_m$ to $G$. This brings the number of independent sets of size $k$ up to $m^m + m^{m-1}$, and it adds
$$
km^{k-1}m^{m-k-1} \leq m^{m-1}
$$
independent sets of size $k-1$. The result is a graph $G'$ with $i_m(G')=\cdots=i_{k+1}(G')=m^m$, $i_k(G')=m^m+m^{m-1}$, with $(i_{k-1}(G'), \ldots, i_1(G'))$ strictly decreasing, with $i_{k-1}(G') \leq (m^m-m^{m-1})+m^{m-1} = m^m$, and with $i_j(G)$ a multiple of $m^j$. The iterative procedure described above (for the construction of $G_m$) can now be continued to obtain $H_k$.

\medskip

\noindent {\bf Case 3}, $k=m$: Instead of starting the construction with $mK_m$, we start with $K_{2m} \cup (m-1)K_m$. This has $2m^m$ independent sets of size $m$, and for $1 \leq k \leq m-1$ it has
$$
\binom{m-1}{(m-1)-k}m^k + 2m\binom{m-1}{m-k}m^{k-1}
$$
independent sets of size $k$ (first consider those without a vertex from the $K_{2m}$, and then those with such a vertex). 

Now consider the sequence $(i_k(K_{2m} \cup (m-1)K_m)_{k=m-1}^1$. The successive terms are integer multiples of $m^{m-1}, m^{m-2}, \ldots$, and the first term is 
$$
m^{m-1}+2m(m-1)m^{m-2} = 2m^m-m^{m-1}.
$$ 
By applying Lemma \ref{lem-obsv-dec} (with $(n,\ell)=(m,m-1)$) to the sequence $(\binom{m-1}{(m-1)-k}m^k)_{k=m-1}^1$, and again (still with $(n,\ell)=(m,m-1)$) to the sequence $(\binom{m-1}{m-k}m^{k-1})_{k=m-1}^1$, we get further that the sequence $(i_k(K_{2m} \cup (m-1)K_m)_{k=m-1}^1$ is strictly decreasing. The iterative procedure described above can now be implemented to obtain $H_m$.

\section{Matching permutations} \label{sec-matperm}

We begin by observing quickly that not all $2^{n-1}$ unimodal permutations of $\{1,\ldots,n\}$ are realizable as the permutation associated to a graph with largest matching $n$. Indeed, the following lemma shows that $m_1(G)$ cannot be the largest entry of a matching sequence of any graph whose largest matching has size at least $4$, so that for $n \geq 4$ the permutation $n(n-1)\cdots321$ is not realizable. (Recall that all graphs under consideration are simple.)
\begin{lemma} \label{lem-2beats1}
If $\nu(G)\geq 4$ then $m_2(G)>m_1(G)$.
\end{lemma}

\begin{proof} 
We proceed by induction on $e(G)$, the number of edges of $G$. In the base case, $e(G)=4$, $G$ must consist of four vertex disjoint edges, and we have $m_2(G)=6>4=m_1(G)$. For the induction step, let $G$ be a graph on more than four edges with $\nu(G)\geq 4$, and let $uv$ be an edge in $G$ (joining vertices $u$ and $v$) chosen so that $G_1$, the graph obtained from $G$ by deleting the edge $uv$, still has a matching with at least four edges. Let $G_2$ be obtained from $G$ by deleting the vertices $u$ and $v$. We have $m_2(G)=m_2(G_1)+m_1(G_2)$ (the set of matchings of size $2$ in $G$ partitions into those that do not include $uv$ --- $m_2(G_1)$ many --- and those that do --- $m_1(G_2)$ many). Also, $m_1(G)=m_1(G_1)+1$. Now by induction $m_2(G_1)>m_1(G_1)$, and also $m_1(G_2)\geq 2>1$, because on deleting $u$ and $v$ from $G$ at least two of the edges of any matching of size $4$ remain. Combining we get $m_2(G)=m_2(G_1)+m_1(G_2) > m_1(G_1)+1=m_1(G)$.
\end{proof}

\medskip

We make an incidental observation at this point. The matching polynomial of a graph with maximum matching size $n$ can be expressed in the form $(1+r_1x)(1+r_2x)\cdots (1+r_nx)$ where the $r_i$'s are real and non-negative; this is a consequence of a theorem of Heilmann and Lieb \cite{HL}. To a sequence that arises as the coefficient sequence of a polynomial of the form $(1+r_1x)(1+r_2x)\cdots (1+r_nx)$ with $r_i$ real and non-negative, we can associate permutations via (\ref{perm-assoc}). Because real-rooted polynomials have unimodal coefficient sequences, at most only the $2^{n-1}$ unimodal permutations of $[n]$ can arise in this context. The permutation $n(n-1)(n-2)\cdots 321$ can arise: let all $r_i$ be equal, say equal to $r$, so the polynomial becomes
$$
1 + \binom{n}{1}rx + \binom{n}{2}r^2x^2 + \cdots + \binom{n}{n-1}r^{n-1}x^{n-1} + r^nx^n.
$$
It's easy to check that if $r$ is sufficiently small, 
$$
r^n < r^{n-1} \binom{n}{n-1} < \cdots < \binom{n}{2}r^2 < \binom{n}{1}r 
$$
so that this polynomial has associated with it the unique permutation $n(n-1)(n-2)\cdots 321$. This shows that our observations about restrictions on the matching sequence are not just restrictions coming in disguise from the real-rooted property of the matching polynomial.

\medskip

The proof of Lemma \ref{lem-2beats1} generalizes considerably. We state and prove the generalization first, and then consider the consequences for matching permutations, in particular giving the proof of Theorem \ref{thm-ub-on-matching}.

\begin{thm} \label{thm-main-res}
For each $n \geq 4$, and for each $k=1,\ldots, \lfloor n/2 \rfloor - 1$, if $\nu(G)\geq n$ then $m_k(G) < m_\ell(G)$ for each $\ell$ satisfying $k < \ell < n-k$.
\end{thm}

\begin{proof} 
We proceed by a double induction, with an outer induction on $n$, and an inner induction on $e(G)$, the number of edges of $G$. The base case of the outer induction, $n=4$, is the assertion that if $\nu(G)\geq 4$ then $m_1(G) < m_2(G)$, which is exactly Lemma \ref{lem-2beats1}.

For $n>4$, assume that we already have the result for all $4 \leq n' < n$. Fix $k$, $1 \leq k \leq \lfloor n/2 \rfloor -1$. We will prove, by induction on $e(G)$, that if $\nu(G) \geq n$ then $m_k(G)<m_\ell(G)$ for any $\ell$ strictly between $k$ and $n-k$. 

In the base case ($e(G)=n$) $G$ must consist of $n$ vertex disjoint edges, and we have $m_\ell(G)=\binom{n}{\ell}>\binom{n}{k}=m_k(G)$. 

For the induction step in this inner induction, let $G$ be a graph on more than $n$ edges, with $\nu(G)\geq n$, and let $uv$ be an edge in $G$, joining vertices $u$ and $v$, chosen so that $G_1$, the graph obtained from $G$ by deleting the edge $uv$, has a matching of size at least $n$. As in the proof of Lemma \ref{lem-2beats1}, let also $G_2$ be obtained from $G$ by deleting the vertices $u$ and $v$. We have
\begin{equation} \label{e0}
m_\ell(G)=m_\ell(G_1)+m_{\ell-1}(G_2)~~\mbox{and}~~m_k(G)=m_k(G_1)+m_{k-1}(G_2).
\end{equation}
Now by the induction hypothesis on $e(G)$, we have
\begin{equation} \label{e1}
m_\ell(G_1) > m_k(G_1).
\end{equation} 
But also, we claim that
\begin{equation} \label{e2}
m_{\ell-1}(G_2) > m_{k-1}(G_2).
\end{equation} 
If $n=5$ then $k=1$ and either $\ell=2$ or $\ell=3$, and (\ref{e2}) becomes either $m_1(G_2) > 1$ (in the case $\ell=1$; note that $m_0(G_2)=1$) or $m_2(G_2) > 1$ (in the case $\ell=2$); both of these hold since $G_2$ has at least three vertex-disjoint edges. For $n > 5$ (\ref{e2}) follows from the $n-2$ case of the of the outer induction. Indeed, $\nu(G_2)\geq n-2$ (removing $u, v$ can delete at most two of the edges from any matching of size $n$). Set $n'=n-2$,  $k'=k-1$ and $\ell'=\ell-1$. We have $1 \leq k \leq \lfloor n/2\rfloor -1$ and $k < \ell < n-k$, so
$0 \leq k-1 \leq \lfloor n/2\rfloor -2$ and $k-1 < \ell-1 < n-k-1$, or 
$0 \leq k' \leq \lfloor n'/2\rfloor -1$ and $k' < \ell' < n'-k'$,
and so the appeal to the earlier case of the outer induction is valid.

Combining (\ref{e1}) and (\ref{e2}) with (\ref{e0}) yields $m_\ell(G) > m_k(G)$, as required.
\end{proof}

\medskip

An immediate consequence of Theorem \ref{thm-main-res} is that for any graph $G$ with $\nu(G)\geq n$ we have $m_{\lfloor n/2\rfloor -1}(G) < m_{\lfloor n/2\rfloor}(G)$, which says that the mode of the matching sequence must occur at $\lfloor n/2\rfloor$ or later. This means that $M_n$, the number of permutations of $[n]$ that can arise as the permutation associated with a graph with largest matching having size $n$, satisfies 
$M_n \leq \sum_{k=\lfloor n/2\rfloor -1}^{n-1} \binom{n-1}{k}$.  
This is asymptotically $2^{n-2}$ as $n$ goes to infinity; a factor of $2$ smaller than the upper bound observed in \cite{AMSE}.

A finer analysis of Theorem \ref{thm-main-res} yields the substantially smaller bound (\ref{eq-up-on-count}) on $M_n$. Let $(m_1, \ldots, m_n)$ be a matching sequence, with mode $m_t$ (perhaps obtained after breaking a tie). Any associated permutation (in one-line notation) puts $\{1, \ldots, t-1\}$ in increasing order and $\{t+1, \ldots, n\}$ in decreasing order in the first $n-1$ spots, and puts $t$ at the end.

This permutation can be encoded by an U-D sequence of length $n-1$ --- each time one sees a U, one enters the first as-yet-unused number from $\{1,\ldots, t-1\}$ (remembering that these numbers should be used in  increasing order); each time one sees a D, one enters the first as-yet-unused number from $\{t+1,\ldots, n\}$ (remembering that these numbers should be used in decreasing order). For example, 
$$
UUDDDUUDUDDUU
$$
would correspond to $n=14, t=8$, and would yield the permutation
$$
1~2~14~13~12~3~4~11~5~10~9~6~7~8.
$$ 
Notice that this is a bijective encoding --- a unique permutation can be read from a sequence. Notice also that in the $U$-$D$ sequence one is never allowed to have an initial substring that has three more $D$'s than $U$'s, because the first time we see such an initial string, say after $j$ $U$'s and $(j+3)$ $D$'s, we would have seen $1$ through $j$, but not $j+1$, and we would have seen $n$ through $n-(j+2)$, in particular including $n-(j+2)$, so we would have
$m_{j+1}>m_{n-(j+2)}$, violating Theorem \ref{thm-main-res}. It follows that $M_n$ is bounded above by the number of $U$-$D$ sequences of length $n-1$ having no initial substring with three more $D$'s than $U$'s. We denote this number by $C^{(3)}_n$. The sequence $(C^{(3)}_n)_{n\geq 1}$ begins $(1, 2, 4, 7, 14, 25, 50, \ldots)$, and is \cite[A026010]{Sloane}.

Rather than deriving an exact formula for $C^{(3)}_n$ (one appears at \cite[A026010]{Sloane}), we take a simpler approach. The quantity $C^{(3)}_n$ is bounded above by the number of $U$-$D$ sequences of length $n+1$ that start with $UU$ and have no initial substring with more $D$'s than $U$'s. This in turn is upper bounded by the number of $U$-$D$ sequences of length $n+1$ having no initial substring with more $D$'s than $U$'s (with no restriction on how the strings start). These sequences are also known as {\em left factors} of Dyck words, and it is well-known (see, for example, \cite[A001405]{Sloane} or \cite[Proposition 1.6]{J-V}) that there are $\binom{n+1}{\lfloor (n+1)/2 \rfloor}$ such. By Stirling's approximation to the factorial, this is asymptotically $c2^n/\sqrt{n}$ (the constant $c$ depending on the parity of $n$). This verifies (\ref{eq-up-on-count}) and completes the proof of Theorem \ref{thm-ub-on-matching}.

\section{Questions and problems} \label{sec-questions}

A number of interesting problems remain concerning the behavior of the independent set sequence of a graph. We begin with the natural refinement of our determination of $f(m)$.
\begin{prob}
For each permutation $\pi$, determine $g(\pi)$, the  minimum order over all graphs $G$ for which $\pi$ is an independent set permutation of $G$.
\end{prob}

\medskip

We have shown that at most $m^m$ vertices is enough to induce the constant weak order on $[m]$ from an independent set sequence, but this is definitely not enough to realize all weak orders; for example, the weak order $m-1 < m < m-2 < m-3 < \cdots < 2 < 1$ requires at least $m^m + m - 2$ vertices. Indeed, if $G$ realizes this weak order, then $i_m(G) > i_{m-1}(G)$, and so, by (the contrapositive of) Theorem \ref{thm-ind-set-main} (Part 2), $i_m(G) \geq m^m$. But we also must have $i_1(G)>i_2(G)>\cdots i_{m-2}(G)>i_m(G)$, so $i_1(G) \geq m^m+m-2$, so $G$ must have at least $m^m+m-2$ vertices. In the other direction, we have shown that fewer than $m^{m+2}$ vertices are sufficient to induce any weak order on $m$.
\begin{prob}
Determine the smallest order large enough to realize every weak order on $[m]$ as the weak order induced by the independent set sequence of some graph.
\end{prob}
\begin{prob} \label{prob-AMSE}
Do the same for weak orders consisting of singleton blocks; equivalently, answer Problem \ref{AMSE-ind-Q} with the additional constraint that the permutations associated with independent set sequences are required to be unique.
\end{prob}
As discussed in the introduction, it is quite likely that the authors of \cite{AMSE} were thinking of Problem \ref{prob-AMSE} when they formulated Problem \ref{AMSE-ind-Q}.

\medskip

A fascinating question is raised in \cite{AMSE}, that has attracted some attention, but has remained mostly open. Although the independent set sequence of a graph is unconstrained, if we restrict to special classes of graphs, then it can become constrained. For example the independent set sequence of a claw-free graph is unimodal \cite{Hamidoune}, and so at most only the $2^{m-1}$ unimodal permutations of $[m]$ can arise as the independent set permutation of a claw-free graph with largest independent set size $m$. Alavi et al. observed that the independent set sequences of stars and paths are both unimodal, and asked:
\begin{question}
\cite[Problem 3]{AMSE} Is the independent set sequence of every tree unimodal?
\end{question}
It is for all trees on 24 or fewer vertices \cite{Radcliffe, YosefMizrachiKadrawi}. See, for example, \cite{GalvinHilyard} for recent work and other references.  

\medskip

It had been conjectured by Levit and Mandrescu  \cite{LM} that every bipartite graph has unimodal independent sequence, and they obtained a partial result: if $G$ is a bipartite graph with $\alpha(G)=m  \geq 1$, then the final third of the independent set sequence is weakly decreasing, i.e.,
$$
i_{\lceil (2m-1)/3 \rceil}(G) \geq \cdots \geq i_{m-1}(G) \geq i_m(G).
$$
The unimodality conjecture was, however, disproved by Bhattacharyya and Kahn \cite{BhattacharyyaKahn}.
\begin{prob} \label{bip-prob}
Characterize the permutations that can occur as the independent set permutations of a bipartite graph.  
\end{prob}
There is an interesting parallel to the case of well covered graphs. A graph is {\em well covered} if all its maximal independent sets have the same size. It had been conjectured by Brown, Dilcher, and Nowakowski \cite{BDN} that every well covered graph has unimodal independent sequence, but this was disproved by Michael and Traves \cite{MT}, who also showed that the first half of the independent set sequence of a well covered graph is increasing, i.e.,
$$
i_1(G) < i_2(G) < \cdots < i_{\lceil m/2 \rceil}(G).
$$
They formulated the {\em roller-coaster} conjecture, that for any $m \geq 1$ and any permutation $\pi$ of $[\lceil m/2 \rceil, m]$ there is a well covered graph $G$ with $\alpha(G)=m$ and with
$$
i_{\pi([\lceil m/2 \rceil)}(G) < i_{\pi([\lceil m/2 \rceil)+1}(G) < \cdots < i_{\pi(m)}(G). 
$$
This was subsequently proved by Cutler and Pebody \cite{CP}. The analog of the roller-coaster conjecture does not hold for Problem \ref{bip-prob}; for example, it is easy to see that for $n \geq 7$, any bipartite graph $G$ on $n$ vertices has $i_2(G) > i_1(G)$. 

\medskip

Turning to matching permutations, the incidental observation made after the proof of Lemma \ref{lem-2beats1} raises the following (perhaps easy) question.
\begin{question}
Which unimodal permutations of $[n]$ can arise via (\ref{perm-assoc}) from the coefficient sequence of a polynomial of the form $(1+r_1x)(1+r_2x)\cdots (1+r_nx)$ with $r_i$ real and non-negative?
\end{question}
 
 \medskip
 
Finally, the greater part of Problem \ref{AMSE-matching-q} remains open.
\begin{prob}
Characterize the permutations that can occur as the matching permutation of a graph,
and determine the growth rate of $M_n$, the number of permutations of $[n]$ that are matching permutations of some graph.
\end{prob}

\medskip

\noindent {\bf Acknowledgement}: We thank the referees for their careful reading and helpful suggestions on presentation.

\end{document}